\documentclass[12pt,reqno]{amsart}

\newtheorem{theorem}{Theorem}[section]

\newtheorem{lemma}[theorem]{Lemma}

\newtheorem{corollary}[theorem]{Corollary}
\usepackage[top=1.2 in, bottom=1.2 in, left=1 in, right=1 in]{geometry}
\newtheorem{definition}[theorem]{Definition}

\newtheorem{question}[theorem]{Question}
\theoremstyle{remark}

\newtheorem{remark}[theorem]{Remark}

\numberwithin{equation}{section}

\usepackage{times}
\usepackage{enumerate}
\usepackage{graphicx,adjustbox}
\usepackage{hyperref}
\usepackage{amsmath,amssymb}
\usepackage{amscd}
\usepackage{graphicx}
\usepackage[all]{xy}
\usepackage{mathtools}

\usepackage{xcolor}

\title[On Nearly Gorenstein Simplicial  Semigroup Algebras]
{On Nearly Gorenstein Simplicial  Semigroup Algebras}
\author{
		Pranjal Srivastava
}
\date{}
\address{\small \rm  Discipline of Mathematics, IISER Bhopal, Madhya Pradesh, India.} 
\email{pranjal.srivastava194@gmail.com}\thanks{The author thanks IISER Bhopal for the Institute post-doc fellowship IISERB/DoFA/PDF/2023/80.}

\date{}



\begin{document}
	
	\begin{abstract}
	In this paper, we study the nearly Gorenstein projective closure of numerical semigroups. We also studied the nealy Gorenstein property of associated graded ring of simplicial affine semigroups. In addition, in case of gluing of numerical semigroups, we answer the question posed by Herzog-Hibi-Stamate (Semigroup Forum 103:550–566, 2021).	
	\end{abstract}

	\maketitle

	\section{Introduction}
	
	Let $\mathbb{N}$ denotes the set of non negative integers and $\mathbb{K}$ 
	denotes a field. 
	Let $(A,\mathfrak{a},\mathbb{K})$ be a positively graded $\mathbb{K}$-algebra which is Cohen-Macaulay and possesses a canonical module $\omega_A$. We recall that if $N$ is any $A-$module, its trace is the ideal $\mathrm{tr}_A(N)=\sum_{\phi \in \mathrm{Hom}_A(N,A)}\phi(N)$ in $A$.  In \cite{Trace}, the ring $A$ is called nearly Gorenstein when $\mathfrak{a} \subseteq\mathrm{tr}_{A}(\omega_{A}) $ and the residue of $A$ is defined as the length of a module $A/\mathrm{tr}_A(\omega_{A})$.
	The trace ideal of $\omega_A$ has been used to define the class of nearly Gorenstein rings. It is well known that $\mathrm{tr}(\omega_A)$ defines the non-Gorenstein locus of $A$ (See \cite[Lemma 2.1]{Trace}). In this paper, we study the trace ideal of canonical module of some simplicial affine semigroups.
	
	Numerical semigroup rings are one-dimensional domains and hence Cohen-Macaulay; these are 
	the coordinate rings of affine monomial curves. A 
	lot of interesting studies have been undertaken 
	by several authors from the viewpoints of singularities, homology and also from purely 
	semigroup theoretic aspects. In \cite{Residue}, Herzog et al. studied the trace of the canonical ideal of numerical semigroup rings and characterized the nearly Gorenstein numerical semigroup using the residue of numerical semigroups.  Moscariello and Strazzanti also studies the nearly Gorenstein affine monimial curves (See \cite{Nearly-Almost}). 
	The projective closure of a numerical semigroup ring is defined by an 
	affine semigroup in $\mathbb{N}^2$ (See Section \ref{PCN}). The Cohen-Macaulay property and many other properties of 
	numerical semigroup rings are not preserved under the operation of projective closure. 
	In this paper, we explore the trace of the canonical ideal of projective closure of numerical semigroup rings. We study the nearly Gorenstein property of projective closure of numerical semigroup whenever numerical semigroups are nearly Gorenstein using Gr\"{o}bner basis of the defining ideal of projective closure of numerical semigroup rings. Recently, Miyashita studied the nearly Gorenstein projective monomial curve for small co-dimensions (See \cite{Miyashita}).

	For a numerical semigroup $\Gamma$, we prove that projective closure $\mathbb{K}[\overline{\Gamma}]$ of numerical semigroup ring $\mathbb{K}[\Gamma]$ is nearly Gorenstein if and only if $\mathbb{K}[\Gamma]$ is nearly Gorenstein under some suitable Gr\"{o}bner basis condition on defining ideal of $\mathbb{K}[\Gamma]$. An important class 
	of numerical semigroups supporting this result is the one defined by an arithmetic sequence; in other words, if $\Gamma$ is a numerical semigroup, minimally generated by an arithmetic sequence, then $\mathbb{K}[\overline{\Gamma}]$ is nearly Gorenstein (See Corollary \ref{BPAS}).
	
	Affine simplicial semigroups provide a natural generalization of numerical
	semigroups, whose theory has been developed mainly in connection to the study of curve singularities. Many authors 
	have studied the properties of the affine semigroup ring $\mathbb{K}[\Gamma]$ from 
	the properties of the affine semigroup $\Gamma$; see 
	\cite{Bruns-Herzog},\cite{Affine}. 
	In this paper, we study the nearly Gorenstein property of the associated graded ring of simplicial affine semigroups. We prove that $\mathrm{gr}_{\mathfrak{m}}(\mathbb{K}[\Gamma])$ is nearly Gorenstein if and only $\mathbb{K}[\Gamma]$ is nearly Gorenstein (See Theorem \ref{CM}).
	
	\smallskip
	
	Herzog et al. (See \cite{Residue}) defined the canonical trace ideal and residue for numerical semigroup rings. They defined the residue of numerical semigroup $\Gamma$, which measures how far is $\Gamma$ from being symmetric. They asked the following questions. 
	
	\begin{question}\label{Ques}(\cite[Question 2.3]{Residue})
		Let $A=\mathbb{K}[\Gamma]$, is it true that 
		\begin{align}
			l(A/\mathrm{tr}_A(\omega_{A}))\leq 
			g(\Gamma)-n(\Gamma)?
		\end{align}
		where $g(\Gamma)=|\mathbb{N}\setminus \Gamma|, n(\Gamma)=|\{s \in \Gamma : s < \mathrm{F}(\Gamma)\}|$ denote
		the number of gaps and the number of non-gaps, respectively, and $\mathrm{F}(\Gamma)$
		denotes the Frobenius number of $\Gamma$, that is, the largest integer in $\mathbb{N}\setminus \Gamma$.
	\end{question}
	
	They show that Question \ref{Ques} holds for $3$-generated numerical semigroup  (See \cite[Proposition 3.2]{Residue}). Recently, Herzog and Kumashiro in \cite{Upper bound} give the answer to this question for all numerical semigroups $\Gamma $ whose $\mathrm{type}(\Gamma)\leq 3$. They also give one example where this question has a negative answer (See \cite[Example 2.5]{Upper bound}). In this paper, we discuss this question for gluing numerical semigroups (See Definition \ref{Gluing-Def}).  	
	We show that Question \ref{Ques} holds for the gluing of numerical semigroups provided glued numerical semigroups satisfy  Question \ref{Ques}. Rosales \cite{JCR} introduced 
	the concept of gluing of numerical semigroups, which was 
	motivated by Delorme's work on 
	complete intersection numerical semigroups. This technique 
	was used extensively to create many examples 
	of set-theoretic complete intersection and ideal-theoretic 
	complete intersection of affine and projective varieties. 
	Feza Arslan et al. \cite{GHM} introduced the concept of 
	\textit{nice gluing} to 
	give infinitely many families of $1$-dimensional local 
	rings with non-Cohen 
	Macaulay tangent cone and non-decreasing Hilbert function. An answer to the Question \ref{Ques} in case of gluing would help us create 
	a family of numerical semigroups in any embedding dimensions in which Question \ref{Ques} holds.
	
	\c{S}ahin in \cite{Sahin} define an operation called ``lifting of monomial curves" (See Subsection \ref{lift} for definition) and use for spreading
	a special property of a monomial curve within an infinite family of examples. In this paper, we prove that $\Gamma$ satisfies inequality (1.1), then its $k$-lifting $\Gamma_k$ also satisfy inequality (1.1). 
	

	\section{Nearly Gorenstein property of projective closure of numerical semigroups}\label{PCN}
	
		Let us  define the projective closure of numerical semigroup rings. 
	Let 
	$e\geq 3$ and $\mathbf{\underline n} = (n_{1}, \ldots, n_{e})$ be a sequence of 
	$e$ distinct positive integers with $\gcd(\mathbf{\underline n})=1$. Let us assume 
	that the numbers $n_{1}, \ldots, n_{e}$ generate the numerical semigroup 
	$\Gamma(n_1,\ldots, n_e) = \langle n_{1}, \ldots , n_{e} \rangle = 
	\lbrace\sum_{j=1}^{e}z_{j}n_{j}\mid z_{j}\in \mathbb{N}\rbrace$ 
	minimally, that is, if $n_i=\sum_{j=1}^{e}z_{j}n_{j}$ for some non-negative 
	integers $z_{j}$, then $z_{j}=0$ for all $j\neq i$ and $z_{i}=1$. We often write 
	$\Gamma$ in place of $\Gamma(n_1,\ldots, n_e)$, when there is no confusion regarding 
	the defining sequence $n_{1}, \ldots, n_{e}$. Let 
	$\eta: R = \mathbb{K}[x_1,\,\ldots,\, x_e]\rightarrow \mathbb{K}[t]$ be the mapping defined by 
	$\eta(x_i)=t^{n_i},\,1\leq i\leq e$. The ideal $\ker (\eta) =I(\Gamma)$  is called 
	the defining ideal of $\Gamma(n_1,\ldots, n_e)$ and it defines the 
	affine monomial curve 
	$\{(u^{n_{1}},\ldots , u^{n_{e}})\in \mathbb{A}^{e}_{\mathbb{K}}\mid u\in \mathbb{K}\} =: 
	C(n_{1},\ldots,n_{e})$ (or simply $C(\Gamma)$). We write 
	$\mathbb{K}[x_{1},\ldots,x_{e}]/I(\Gamma) =: \mathbb{K}[\Gamma(n_1,\ldots, n_e )]$ 
	(or simply $\mathbb{K}[\Gamma]$), which is called the semigroup ring for 
	the semigroup $\Gamma(n_1,\ldots, n_e )$. It is known that $I(\Gamma)$ is 
	generated by the binomials $x^{a}-x^{b}$, where $a$ and $b$ are $e$-tuples of non-negative 
	integers with $\eta(x^{a})=\eta(x^{b})$.

	
	Let $n_{e}>n_{i}$ for all $i<e$, and $n_{0}=0$. We define the 
	semigroup $\overline{\Gamma(n_{1},\ldots n_{e})}  
	= \langle \{(n_{i},n_{e}-n_{i})\mid 0\leq i\leq e\} \rangle  = \lbrace \sum_{i=0}^{e}z_{i}(n_{i},n_{e}-n_{i}) \mid 
	z_{i}\in\mathbb{N}\rbrace$, often written as $\overline{\Gamma}$. 
	Let $\eta^{h}:S=\mathbb{K}[x_{0},\ldots,x_{e}]\longrightarrow \mathbb{K}[s,t]$ 
	be the $\mathbb{K}$-algebra map defined as $\eta^{h}(x_{i})=t^{n_{i}}s^{n_{e}-n_{i}}, 0\leq i \leq e$  and 
	$ \ker(\eta^{h}) = \overline{I(\Gamma)}$. The homogenization of 
	$I(\Gamma)$ with respect to the variable $x_{0}$ is 
	$\overline{I(\Gamma)}$, which defines 
	$\{[(v^{n_{e}}: v^{n_{e}-n_{1}}u^{n_{1}}: \cdots: u^{n_{e}})]\in\mathbb{P}^{e}_{\mathbb{K}}\mid u, v\in \mathbb{K}\} =: \overline{C(n_{1},\ldots,n_{e})}$ (or simply $\overline{C(\Gamma)}$),  
	the projective closure of the affine monomial curve $C(n_{1},\ldots,n_{e})$. 
	The $\mathbb{K}$-algebra 
	$\mathbb{K}[x_{0},\ldots,x_{e}]/\overline{I(\Gamma)} =: 
	\mathbb{K}[\overline{\Gamma(n_1,\ldots, n_e )}]$ (or simply $\mathbb{K}[\overline{\Gamma}]$) 
	denotes the coordinate ring. It can be proved easily that 
	$\overline{C(n_{1},\ldots,n_{e})}$ is a projective curve, which is said to be 
	arithmetically Cohen-Macaulay if the vanishing ideal $\overline{I(\Gamma)}$ 
	is a Cohen-Macaulay ideal.

	\begin{lemma}\label{Can-Exist}
		Let $(R,\mathfrak{m})$ be a Cohen-Macaulay graded ring with canonical module $\omega_{R}$. Let $\phi:(R,\mathfrak{m}) \rightarrow (S,\mathfrak{n})$ be a ring homomorphism of Cohen-Macaulay graded rings satisfying
		\begin{enumerate}[(i)]
			\item $\phi(R_i)\subset S_i$ for all $i \in \mathbb{Z}$,
			\item  $\phi(\mathfrak{m})\subset \mathfrak{n}$,
			\item $S$ is a finite graded $R$-module.
		\end{enumerate}
		Then $\omega_{S} \cong \mathrm{Ext}^{t}_{R}(S,\omega_{R})$, where $t=\mathrm{dim}(R)-\mathrm{dim}(S)$.
	\end{lemma}
	\begin{proof}
		See Proposition 3.6.12 in \cite{Bruns-Herzog}
	\end{proof}

	\begin{remark}
		Note that a graded polynomial ring $R[x_0]=\mathbb{K}[x_0,\dots,x_e]$ over a field $\mathbb{K}$ is Gorenstein, hence \cite[Proposition 3.6.11]{Bruns-Herzog} $\omega_{R[x_0]}\cong R[x_0]$. Suppose  $\mathbb{K}[\overline{\Gamma}]$ is Cohen-Macaulay then by Lemma \ref{Can-Exist} $\mathbb{K}[\overline{\Gamma}]$ admits a canonical-module $\omega_{\mathbb{K}[\overline{\Gamma}]}$ and $\omega_{\mathbb{K}[\overline{\Gamma}]}\cong \mathrm{Ext}_{R[x_0]}^{e-1}(R[x_0]/I^h,R[x_0])$.   
	\end{remark}

	\begin{theorem}\label{BSP}
		Let $\Gamma$ be a numerical semigroup, such that $\mathbb{K}[\overline{\Gamma}]$ is arithmetically Cohen-Macaulay. 
		If there exists a Gr\"{o}bner basis $G$ of the defining ideal $I(\Gamma)$ of $\mathbb{K}[\Gamma]$, 
		with respect to the degree reverse lexicographic ordering $x_{i}>x_{e}$, for $i=1,\dots,e-1$, such that $x_{e}$ belongs to 
		the support of all non-homogeneous elements of $G$. Then $\mathbb{K}[\Gamma]$ is nearly-Gorenstein if and only if 
		$\mathbb{K}[\overline{\Gamma}]$ is nearly-Gorenstein.															
	\end{theorem}																																	
	\begin{proof} Let $G=\{f_{1},\dots,f_{l},g_{1},\dots,g_{r}\}$ be a 
		Gr\"{o}bner basis of $I(\Gamma)$, with respect to 
		the degree reverse lexicographic order $x_{i}>x_{e}$. Let 
		$f_{1},\dots, f_{l}$ be homogeneous binomials and $g_{1},\dots, g_{r}$ 
		be non-homogeneous elements. It follows from Lemma 2.1 in \cite{Herzog-Stamate} 
		that $G^{h}$ is a Gr\"{o}bner basis of $\overline{I(\Gamma)}$, 
		with respect to the monomial order induced by $x_{i}>x_{0}$. Since 
		$\mathbb{K}[\overline{\Gamma}]$ is arithmetically Cohen-Macaulay, $x_{e}$ does 
		not divide any leading monomial of $G$. From our assumption, the terms 
		of $g_{i}$, which are not the leading term is divisible by $x_{e}$, 
		for $i=1,\dots,e$. 
		
		Consider the natural map $\pi : R[x_0]= \mathbb{K}[x_{0},\dots,x_{e}] \longrightarrow \mathbb{K}[x_{1},\dots,x_{e-1}]$, 
		given by $\pi(x_{e})=0,\pi(x_{0})=0$ and $\pi(x_{i})=x_{i}$ for $i=1,\dots,e-1$. Note that $\frac{\mathbb{K}[x_{0},\dots,x_{e}]}{(G^{h},x_{0},x_{e})} \cong \frac{\mathbb{K}[x_{1},\dots,x_{e-1}]}{\pi(G^{h})}$ as $S/x_{0}$-module. 
		Since $\pi(g_{i})$ is a monomial which is not divisible by $x_{0}$, $x_{e}$ and 
		since any term of $f_{i}$ is not divisible by $x_{0}$, we have $\pi(G^{h})=\pi(G).$ 
		Now, $x_{e}$ being regular in $\frac{\mathbb{K}[x_{0},\dots,x_{e}]}{G^{h}}$, we have 
		$$\omega_{\mathbb{K}[\overline{\Gamma}]}\cong \mathrm{Ext}_{R[x_0]}^{e-1}(R[x_0]/G^h,R[x_0])= \mathrm{Ext}_{R[x_0]}^{e-1}\left(\frac{\mathbb{K}[x_{0},\dots,x_{e}]}{(G^{h},x_{e})}, R[x_0]\right).$$
		Moreover, since $x_{0}$ is regular in $S/(x_{e},G^{h})$, we can write
		\begin{align*}
			\mathrm{Ext}_{R[x_0]}^{e-1}\left(\frac{\mathbb{K}[x_{0},\dots,x_{e}]}{(G^{h},x_{e})}, R[x_0]\right) \cong&
			\mathrm{Ext}_{R}^{e-1}\left(\frac{\mathbb{K}[x_{0},\dots,x_{e}]}{(G^{h},x_{e},x_0)}, R[x_0]\right) \\ \cong& 
			\mathrm{Ext}_{R}^{e-1}\left(\frac{\mathbb{K}[x_{1},\dots,x_{e-1}]}{(\pi(G^{h}))}, R[x_0]\right)\\ \cong& 		\mathrm{Ext}_{R}^{e-1}\left(\frac{\mathbb{K}[x_{1},\dots,x_{e-1}]}{(\pi(G))}, R[x_0]\right) \\ \cong&\mathrm{Ext}_{R}^{e-1}\left(\frac{\mathbb{K}[x_{1},\dots,x_{e}]}{(G,x_e)}, R[x_0]\right)\\ \cong &\mathrm{Ext}_{R}^{e-1}\left(\frac{\mathbb{K}[x_{1},\dots,x_{e}]}{G}, R[x_0]\right)\cong \omega_{\mathbb{K}[\Gamma]}R[x_0].
		\end{align*}
		which is true since 
		$x_{e}$ is both $\mathbb{K}[x_{1},\dots,x_{e}]$-regular and $\mathbb{K}[x_{1},\dots,x_{e}]/G$ - regular. Hence $\omega_{\mathbb{K}[\overline{\Gamma}]}\cong \omega_{\mathbb{K}[\Gamma]}R[x_0] $.  Since every elements of $\mathrm{Hom}_{R[x_0]}(\omega_{\mathbb{K}[\overline{\Gamma}]},R[x_0])$ can be obtained  obtained by elements of $ \mathrm{Hom}_{R}(\omega_{\mathbb{K}[\Gamma]},R)$ using composition with $f:R \rightarrow R[x_0]$. Hence, $$\mathrm{tr}_{R[x_0]}(\omega_{\mathbb{K}[\overline{\Gamma}]})=\sum_{\phi \in \mathrm{Hom}_{R[x_0]}(\omega_{\mathbb{K}[\overline{\Gamma}]},R[x_0])}\phi(\omega_{\mathbb{K}[\overline{\Gamma}]})\cong\sum_{\phi \in \mathrm{Hom}_{R}(\omega_{\mathbb{K}[\Gamma]},R)}\phi(\omega_{\mathbb{K}[\Gamma]})=\mathrm{tr}_{R}(\omega_{\mathbb{K}[\Gamma]})R[x_0].$$
		Therefore, $\mathbb{K}[\Gamma]$ is nearly-Gorenstein implies  
		$\mathfrak{m} \subseteq \mathrm{tr}_{R}(\omega_{\mathbb{K}[\Gamma]}) $ and by above equality $\mathfrak{m}[x_0] \subseteq  \mathrm{tr}_{R[x_0]}(\omega_{\mathbb{K}[\overline{\Gamma}]})$, therefore $\mathbb{K}[\overline{\Gamma}]$ is nearly-Gorenstein and vice versa.
	\end{proof}
	
	\medspace
	
	\begin{corollary}\label{BPAS}
		Let $\Gamma=\Gamma(n_{1},\dots,n_{e})$ be a numerical semigroup, minimally generated by 
		an arithmetic sequence $n_{1}<n_{2}<\dots<n_{e}$, such that $n_{i}=n_{1}+(i-1)d, 1 \leq i \leq e, e\geq n_1$  
		and $n_{1}=q(e-1)+r', r' \in [1,e]$. Then $\mathbb{K}[\overline{\Gamma}]$ is nearly Gorenstein.	
	\end{corollary}		
	
	\begin{proof} The set $G=\{x_{i}x_{j}-x_{i-1}x_{j+1} \vert 2 \leq i \leq  j \leq n-1\} \cup \{x_{1}^{q+d}x_{i}-x_{e+i}x_{e}^{q} \vert 1 \leq i \leq e-r'\}$ is a Gr\"{o}bner basis of the defining ideal $I(\Gamma)$ of $C(\Gamma)$, with respect to the degree reverse lexicographic ordering 
		induced by $x_{1}>x_{2}>\cdots>x_{e}$ and 
		$in_{<}(G)=\{x_{i}x_{j} \vert 2 \leq i \leq j \leq e-1\} \cup \{x_{1}^{q+d}x_{i} \vert 1 \leq i \leq e-r'\}$(See \cite{Bermejo}). Note that, 
		$x_{e}$ belongs to the support of all non-homogeneous elements of $G$. 
		By Theorem \ref{BSP}, and \cite[Proposition 2.4]{Residue}, $\mathbb{K}[\overline{\Gamma}]$ is nearly Gorenstein.	
	\end{proof}
	
	\medspace
	
	\section{Nearly Gorenstein Associated graded ring of simplicial Affine semigroups}
	\medskip
	
	Let $\Gamma$ be an affine semigroup, i.e, a finitely generated semigroup which for some $r$ is isomorphic to a subsemigroup of $\mathbb{Z}^r$ containing zero. 
 Suppose that 
	$\Gamma$ is a simplicial affine semigroup, fully embedded in $\mathbb{N}^{d}$, minimally generated by 
	$\{\mathbf{a}_{1}.\dots,\mathbf{a}_{d},\mathbf{a}_{d+1},\dots,\mathbf{a}_{d+r}\}$ with the set 
	of extremal rays $E=\{\mathbf{a}_{1},\dots,\mathbf{a}_{d}\}$. 	The semigroup algebra $\mathbb{K}[\Gamma]$ over a field $\mathbb{K}$ is generated by the 
	monomials $\mathbf{x}^{\mathbf{a}}$, where $\mathbf{a} \in \Gamma$, with maximal ideal 
	$\mathfrak{m}=(\mathbf{x}^{\mathbf{a}_{1}},\dots,\mathbf{x}^{\mathbf{a}_{d+r}})$.
	
	Let $I(\Gamma)$ denote the defining ideal of $\mathbb{K}[\Gamma]$, which is the 
	kernel of the $\mathbb{K}-$algebra homomorphism  
	$\phi:R=\mathbb{K}[z_{1},\dots,z_{d+r}] \rightarrow \mathbb{K}[\mathbf{x}^{\mathbf{a}_{1}},\dots,\mathbf{x}^{\mathbf{a}_{d+r}}]$, 
	such that $\phi(z_{i})=\mathbf{x}^{\mathbf{a}_{i}}$, $i=1,\dots,d+r$. Let 
	us write $\mathbb{K}[\Gamma]\cong A/I(\Gamma)$.  The defining ideal $I(\Gamma)$ 
	is a binomial prime ideal (\cite{Herzog}, Proposition 1.4). 
	The associated graded ring 
	$\mathrm{gr}_{\mathfrak{m}}(\mathbb{K}[\Gamma])=\oplus_{i=0}^{\infty}\mathfrak{m}^{i}/\mathfrak{m}^{i+1}$ is 
	isomorphic to $\frac{\mathbb{K}[z_{1},\dots,z_{d+r}]}{I(\Gamma)^{*}}$ 
	(see \cite{Bruns-Herzog}, Example 4.6.3), where $I(\Gamma)^{*}$ is the 
	homogeneous ideal generated by the initial forms $f^{*}$ of the elements 
	$f\in I(\Gamma)$, and $f^{*}$ is the homogeneous summand of $f$ of the 
	least degree.

		\begin{definition}{\rm 
			Let $\Gamma$ be an affine semigroup in $\mathbb{N}^{d}$, minimally generated by $\mathbf{a}_{1},\dots,\mathbf{a}_{n}$. The \textit{rational polyhedral cone} generated by $\Gamma$ is defined as
			\[
			\mathrm{cone}(\Gamma)=\big\{\sum_{i=1}^{n}\alpha_{i}\mathbf{a}_{i}: \alpha_{i} \in \mathbb{R}_{\geq 0}, \,i=1,\dots,n\big\}.
			\]
			The \textit{dimension} of $\Gamma$ is defined as the dimension of the subspace generated by $\mathrm{cone}(\Gamma)$.
		}
	\end{definition} 
	
	The $\mathrm{cone}(\Gamma)$ is the intersection of finitely many closed linear half-spaces in $\mathbb{R}^{d}$, 
	each of whose bounding hyperplanes contains the origin. These half-spaces are called \textit{support hyperplanes}. 
	
	\begin{definition}\label{Extremal rays} {\rm 
			Suppose $\Gamma$ is an affine semigroup of dimension $r$ in $\mathbb{N}^{d}$. If $r=1$, $\mathrm{cone}(\Gamma)=\mathbb{R}_{\geq 0}$.
			If $r=2$, the support hyperplanes are one-dimensional vector spaces, which are called the 
			\textit{extremal rays} of $\mathrm{cone}(\Gamma)$. If $r >2$, intersection of any two
			adjacent support hyperplanes is a one-dimensional vector space, called an extremal ray 
			of $\mathrm{cone}(\Gamma)$. An element of $\Gamma$ is called an extremal ray of $\Gamma$ 
			if it is the smallest non-zero vector of $\Gamma$ in an extremal ray of $\mathrm{cone}(\Gamma)$. 
		}
	\end{definition}
	
	\begin{definition}\label{Extremal}{\rm 
			An affine semigroup $\Gamma$ in $\mathbb{N}^{d}$, is said to be \textit{simplicial}   
			if the $\mathrm{cone}(\Gamma)$ has exactly $d$ extremal rays, i.e., if there exist a set with $d$ elements, say 
			$\{\mathbf{a}_{1},\dots,\mathbf{a}_{d}\} \subset \{\mathbf{a}_{1},\dots,\mathbf{a}_{d},\mathbf{a}_{d+1},\dots,\mathbf{a}_{d+r}\}$, such that they 
			are linearly independent over $\mathbb{Q}$ and $\Gamma \subset \sum\limits_{i=1}^{d}\mathbb{Q}_{\geq 0}\mathbf{a}_{i}$.
		}
	\end{definition}

	\begin{definition}{\rm
			Let $(B,\mathcal{F})$ be a filtered, Noetherian ring. A sequence $g = g_{1},\dots,g_{n}$ in $B$ is called 
			\textit{super regular} if the sequence of initial forms $g^{*} = g_{1}^{*},\dots,g_{n}^{*}$ is 
			regular in $\mathrm{gr}_{\mathcal{F}}(B)$.
		}
	\end{definition}
	
	\begin{lemma}\label{Cond}
		Let 
		$(\mathbf{x}^{\mathbf{a}_{1}},\dots,\mathbf{x}^{\mathbf{a}_{d+r}})$ be a reduction ideal of $\mathfrak{m}$, then the following statements are equivalent:
		\begin{enumerate}[(a)]
			\item $\mathrm{gr}_{\mathfrak{m}}(\mathbb{K}[\Gamma])$ is a Cohen-Macaulay ring.
			\item $(\mathbf{x}^{\mathbf{a}_{1}})^{*},\dots,(\mathbf{x}^{\mathbf{a}_{d}})^{*}$ provides a regular sequence in $\mathrm{gr}_{\mathfrak{m}}(\mathbb{K}[\Gamma])$.
			\item $\mathbb{K}[\Gamma]$ is Cohen-Macaulay and $(\mathbf{x}^{\mathbf{a}_{i}})^{*}$ is a non-zero divisor in 
			$\mathrm{gr}_{\mathfrak{m}}(\mathbb{K}[\Gamma]),\, \text{for} \,\, i=1,\dots,d$.
		\end{enumerate}
	\end{lemma}
	
	\begin{proof} See Proposition 5.2 in \cite{Reduction}. 
	\end{proof}

	Let $A$ be a filtered noetherian graded ring with homogeneous maximal ideal $\mathfrak{m}_{A}$ and suppose $B=A/xA$, where $x$ is not a zero-divisor on $A$. Let 
	$\psi:A \rightarrow B$ be the canonical epimorphism.

	\begin{lemma}\label{iso}
		If $x$ is a super regular in $A$ then 
		\[
		\mathrm{gr}_{\mathfrak{m}_{A}}(A) \xrightarrow{(x)^{*}} \mathrm{gr}_{\mathfrak{m}_{A}}(A) \xrightarrow{\mathrm{gr}(\psi)} \mathrm{gr}_{\mathfrak{m}_{B}}(B) \rightarrow 0
		\]
		is exact. 
	\end{lemma} 
	\begin{proof} See Lemma a in \cite{Super}.
		
	\end{proof}
	
	\begin{lemma}\label{iso1}
		Consider a map $$\pi_{d}:(R:=)\mathbb{K}[z_{1},\dots,z_{d},\dots,z_{d+r}] \rightarrow \bar{R}=\mathbb{K}[z_{d+1},\dots,z_{d+r}]$$ such that $\pi_{d}(z_{j})=0, 1 \leq j \leq d$ and $\pi_{d}(z_{j})=z_{j}, d+1 \leq j \leq d+r$. If $z_{1},\dots,z_{d}$ is a super regular in $R/I(\Gamma)$ then
		
		$$\mathrm{gr}_{\bar{\mathfrak{m}}}\big(\bar{R}/\pi_{d}(I(\Gamma)) \cong \frac{\mathrm{gr}_{\mathfrak{m}}(R/I(\Gamma))}{(z_{1},\dots,z_{d})\mathrm{gr}_{\mathfrak{m}}(R/I(\Gamma))},$$ 
		where  $\bar{\mathfrak{m}}=\pi_{d}(\mathfrak{m})$.
	\end{lemma}
	
	\begin{proof}
		See Lemma 3.8 in \cite{Saha-Associated}.
	\end{proof}
	
	\begin{remark}
		Assume $\mathrm{gr}_{\mathfrak{m}}(\mathbb{K}[\Gamma])$ is Cohen-Macaulay and  $\mathrm{dim}(\mathbb{K}[\Gamma])=d$. Then by Lemma \ref{Can-Exist}, $\mathrm{gr}_{\mathfrak{m}}(\mathbb{K}[\Gamma])$ admits a canonical module $\omega_{\mathrm{gr}_{\mathfrak{m}}(\mathbb{K}[\Gamma])} $ and $\omega_{\mathrm{gr}_{\mathfrak{m}}(\mathbb{K}[\Gamma])} \cong \mathrm{Ext}^{r}_R(\mathrm{gr}_{\mathfrak{m}}(\mathbb{K}[\Gamma]),R)$.
	\end{remark}
	
	\smallskip
	
	\begin{theorem}\label{CM}
		Let $(\mathbf{x}^{\mathbf{a}_{1}},\dots,\mathbf{x}^{\mathbf{a}_{d+r}})$ be a reduction ideal of $\mathfrak{m}$. 
		Suppose $\mathbb{K}[\Gamma]$ and $\mathrm{gr}_{\mathfrak{m}}(\mathbb{K}[\Gamma])$ are Cohen-Macaulay. 
		Let $G=\{f_{1},\dots,f_{t},g_{1},\dots,g_{s}\}$ be a minimal Gr\"{o}bner basis of the defining ideal $I(\Gamma)$, with respect to the negative degree reverse 
		lexicographic ordering induced by $z_{d+r} > \dots > z_{d} > \dots >z_{1}$. 
		We assume that $f_{1}, \dots,f_{t}$ are homogeneous and $g_{1},\dots,g_{s}$ 
		are non homogeneous, with respect to the standard gradation on the polynomial 
		ring $\mathbb{K}[z_{1}, \ldots, z_{d+r}]$. If there exists a $j$, $1\leq j \leq d$, such that 
		$z_{j}$ belongs to the support of $g_{l}$, for every $1 \leq l \leq s $.
		Then $\mathbb{K}[\Gamma]$ is nearly Gorenstein if and only if   $\mathrm{gr}_{\mathfrak{m}}(\mathbb{K}[\Gamma]))$ is also nearly Gorenstein.
	\end{theorem}
	
	\begin{proof} 
		Let $G=\{f_{1},\dots,f_{t},g_{1},\dots,g_{s}\}$ be a minimal Gr\"{o}bner basis of the defining ideal $I(\Gamma)$, with respect to the negative degree reverse 
		lexicographic ordering induced by $z_{d+r} > \dots > z_{d} > \dots >z_{1}$. 
		When $s=0$, $I(\Gamma)$ is homogeneous ideal and from Remark 2.1 (\cite{Reduction}),  
		$\mathbb{K}[\Gamma] \cong \mathrm{gr}_{\mathfrak{m}}(\mathbb{K}[\Gamma])$. Hence, the 
		result follows directly.
		\medskip
		
		When $s\geq 1$, we have $f_{1}, \dots, f_{t}$ are homogeneous, 
		$g_{1},\dots g_{s}$ are non-homogeneous and $\mathrm{gr}_{\mathfrak{m}}(\mathbb{K}[\Gamma])$ 
		is Cohen-Macaulay, this implies that $z_{1},\dots,z_{d}$ do not divide the 
		$\mathrm{LM}(f_{k})$ 
		and $\mathrm{LM}(g_{l})$ for $k=1,\dots, t$, $l=1,\dots, s$. Moreover, 
		$z_{j} \in \mathrm{supp}(\{g_{1},\dots,g_{s}\})$, for some $1 \leq j \leq d$. Therefore, $z_{j}$ divides a non-leading term of $g_{1},\dots,g_{s}$, for 
		some  $1 \leq j \leq d$.
		\medskip
		
		We consider the map $$\pi_{d}:R=\mathbb{K}[z_{1},\dots,z_{d},\dots,z_{d+r}] \rightarrow \bar{R}=\mathbb{K}[z_{d+1},\dots,z_{d+r}]$$ such that $\pi_{d}(z_{j})=0$, 
		$1 \leq j \leq d$ and $\pi_{d}(z_{j})=z_{j}, d+1 \leq j \leq d+r$. 
		We note that $\pi_{d}(f_{1}),\dots,\pi_{d}(f_{t})$ are either monomials 
		or homogeneous polynomials. Since $z_{j}$ divides a non-leading 
		term of $\{g_{1},\dots,g_{s}\}$ for some $1 \leq j \leq d+r$, we must have 
		that $\pi_{d}(g_{1}),\dots,\pi_{d}(g_{s})$ are the leading monomials of $g_{1},\dots,g_{s}$ respectively. 
		\medskip
		
		Therefore $\{\pi_{d}(f_{1}),\dots,\pi_{d}(f_{t}),\pi_{d}(g_{1}),\dots,\pi_{d}(g_{s})\}$ generates the homogeneous ideal $\pi_{d}(I(\Gamma))$. Hence
		\begin{align*}
			\mathrm{Ext}^{r}_{R}\big(\bar{R}/\pi_{d}(I(\Gamma),R)\big)&\cong \mathrm{Ext}^{r}_{R}(\mathrm{gr}_{\bar{\mathfrak{m}}}\big(\bar{R}/\pi_{d}(I(\Gamma)),R\big) 
		\end{align*}
		where  $\bar{\mathfrak{m}}=\pi_{d}(\mathfrak{m})$. Since $\mathbb{K}[\Gamma],\, \mathrm{gr}_{\mathfrak{m}}(\mathbb{K}[\Gamma])$ are Cohen-Macaulay, by Lemma \ref{Cond} $z_{1},\dots,z_{d}$ are regular in $\mathrm{gr}_{\mathfrak{m}}(R/I(\Gamma))$ , and therefore $z_{1},\dots,z_{d}$ form a super regular sequence in $(R/I(\Gamma))$. By 
		Lemma \ref{iso1}, 
		\[\mathrm{gr}_{\bar{\mathfrak{m}}}\big(\bar{R}/\pi_{d}(I(\Gamma)) \cong \frac{\mathrm{gr}_{\mathfrak{m}}(R/I(\Gamma))}{(z_{1},\dots,z_{d})\mathrm{gr}_{\mathfrak{m}}(R/I(\Gamma))},
		\]  
		therefore $$\mathrm{Ext}^{r}_{R}(\mathrm{gr}_{\bar{\mathfrak{m}}}\big(\bar{R}/\pi_{d}(I(\Gamma),R)\big) \cong  \mathrm{Ext}^{r}_{R}\bigg(\frac{\mathrm{gr}_{\mathfrak{m}}(R/I(\Gamma))}{(z_{1},\dots,z_{d})\mathrm{gr}_{\mathfrak{m}}(R/I(\Gamma))},R\bigg).$$
		$\mathrm{gr}_{\mathfrak{m}}(R/I(\Gamma))$ being Cohen-Macaulay, $z_{1},\dots,z_{d}$ form 
		a regular sequence in $\mathrm{gr}_{\mathfrak{m}}(R/I(\Gamma))$, 
		hence
		$$ \mathrm{Ext}^{r}_{\frac{R}{z_1,\dots,z_d}}\bigg(\frac{\mathrm{gr}_{\mathfrak{m}}(R/I(\Gamma))}{((z_{1},\dots,z_{d})\mathrm{gr}_{\mathfrak{m}}(R/I(\Gamma)},R\bigg)\cong\mathrm{Ext}^{r}_{R}\big(\mathrm{gr}_{\mathfrak{m}}(R/I(\Gamma),R\big).$$
		$R/I(\Gamma)$ being Cohen-Macaulay, $z_{1},\dots,z_{d}$ form a regular sequence 
		in $R/I(\Gamma)$. Hence,
		
		the canonical module of $\mathrm{gr}_{\mathfrak{m}}(R/I(\Gamma)$ is given by
		\begin{align*}
			\omega_{\mathrm{gr}_{\mathfrak{m}}(\frac{R}{I(\Gamma)})} \cong & \mathrm{Ext}^{r}_R(\mathrm{gr}_{\mathfrak{m}}(\frac{R}{I(\Gamma)}),R)\cong  \mathrm{Ext}^{r}_{\frac{R}{z_1,\dots,z_d}}(\frac{\mathrm{gr}_{\mathfrak{m}}(R/I(\Gamma))}{((z_{1},\dots,z_{d})\mathrm{gr}_{\mathfrak{m}}(R/I(\Gamma)},R)\\
			\cong &  \mathrm{Ext}^{r}_{\frac{R}{z_1,\dots,z_d}}(\mathrm{gr}_{\bar{\mathfrak{m}}}\big(\bar{R}/\pi_{d}(I(\Gamma),R)\\
			\cong & \mathrm{Ext}^{r}_{\frac{R}{z_1,\dots,z_d}}((\bar{R}/\pi_{d}(I(\Gamma)),R)\\
			\cong &
			\mathrm{Ext}^{r}_{\frac{R}{z_1,\dots,z_d}}((R/(z_{1},\dots,z_{d},I(\Gamma)),R)\\
			\cong& \mathrm{Ext}^{r}_{R}((R/I(\Gamma)),R)\\
			\cong &
			\omega_{\frac{R}{I(\Gamma)}},
		\end{align*}
		and we have $\mathrm{tr}(\omega_{\frac{R}{I(\Gamma)}})\cong \mathrm{tr}(\omega_{\mathrm{gr}_{\mathfrak{m}}(\frac{R}{I(\Gamma)})})$. Therefore, $\mathrm{gr}_{\mathfrak{m}}(\frac{R}{I(\Gamma)})$ is nearly Gorenstein if and only if $\frac{R}{I(\Gamma)}$ is nearly Gorenstein.

	\end{proof}

	\section{Residue for gluing of numerical semigroups}
	\medspace

	A numerical semigroup $\Gamma$ is a submonoid of $\mathbb{N}$ which have a finite complement in $\mathbb{N}$. Say $\Gamma$ is minimally generated  y $n_1<\dots <n_e$ with $e>1$. We write $\Gamma=\langle n_1,\dots,n_e \rangle$. The number $e$ is called the \emph{embedding dimension} of $\Gamma$. The elements in the set $G(\Gamma)=\mathbb{N}\setminus \Gamma$ is called the gaps of $\Gamma$ and the cardinality of $G(\Gamma)$ is called genus of $\Gamma$, denoted by $g(\Gamma)$. As $|G(\Gamma)|$ is finite, there exists a largest integer $\mathrm{F}(\Gamma)$, called the \emph{Frobenius number} of $\Gamma
	$, such that $\mathrm{F}(\Gamma) \notin \Gamma$. 
	Let $M:=\Gamma \setminus \{0\}$. The element $\nu \in G(\Gamma)$ with $\nu +M \in \Gamma$ are called \emph{pseudo-Frobenius numbers}, denoted by $\mathrm{PF}(\Gamma)$ and its cardinality is known as \emph{type} of $\Gamma$, denoted $\mathrm{type}(\Gamma)$. In semigroup ring $\mathbb{K}[\Gamma]$, the canonical module $\omega_{\mathbb{K}[\Gamma]}$ is the fractionary $\mathbb{K}[\Gamma]$-ideal generated by the elements $t^{\nu}$ with $\nu \in \mathrm{PF}(\Gamma)$, see \cite[Exercise 21.11]{Eisenbud}. Therefore, the Cohen-Macaulay type of $\mathbb{K}[\Gamma]$ is equal to $\mathrm{type}(\Gamma)$. The anti-canonical ideal of $\mathbb{K}[\Gamma]$  is the fractionary ideal $\omega^{-1}_{\mathbb{K}[\Gamma]}=\{x \in Q(\mathbb{K}[\Gamma]): x \cdot \omega_{\mathbb{K}[\Gamma]}\}$, where  $Q(\mathbb{K}[\Gamma])$ is quotient field of $\mathbb{K}[\Gamma]$. Let $\Omega_{\Gamma}$ and $\Omega_{\Gamma}^{-1}$ be the set of exponents of the monomials in $\omega_{\mathbb{K}[\Gamma]}$, and in $\omega_{\mathbb{K}[\Gamma]}^{-1}$ respectively.
	The \emph{trace} of $\Gamma$ as $\mathrm{tr}(\Gamma)=\Omega_{\Gamma}+\Omega_{\Gamma}^{-1}$. It is clear that $\mathrm{tr}(\Gamma)$ is an ideal in $\Gamma$ consisting of the exponents of the monomials in $\mathrm{tr}(\mathbb{K}[\Gamma])$.
	As $\mathbb{K}[\Gamma]/\mathrm{tr}(\omega_{\mathbb{K}[\Gamma]})$ is a finite dimensional vector space with a $\mathbb{K}$-basis given by $\{t^{\gamma}: \gamma \in \Gamma \setminus \mathrm{tr}(\Gamma)\}$. The \emph{residue} of $\Gamma$ is defined as the residue of $\mathbb{K}[\Gamma]$, namely
	$$ \mathrm{res}(\Gamma)=\mathrm{dim}_{\mathbb{K}}\mathbb{K}[\Gamma]/\mathrm{tr}(\omega_{\mathbb{K}[\Gamma]})=|\Gamma \setminus \mathrm{tr}(\Gamma)|.$$
	
	Now, we are ready to answer Question \ref{Ques} in case of gluing of numerical semigroups. Let us first recall  Question \ref{Ques} and definition of gluing of numerical semigroups. 
	
	\medspace
	
	\noindent \textbf{Question \ref{Ques}}
	Given a numerical semigroup $\Gamma$, is it true that $$\mathrm{res}(\Gamma)\leq n(\Gamma)-g(\Gamma)?$$

\smallskip

\begin{definition}[\cite{JCR}]\label{Gluing-Def} 
	{\rm 
		Let $\Gamma_{1} = \Gamma(m_{1},\dots, m_{l})$ and 
		$\Gamma_{2} = \Gamma(n_{1},\dots, n_{k})$ be two numerical semigroups, 
		with $m_{1} < \cdots < m_{l}$ and $n_{1} < \cdots < n_{k}$. 
		Let $\lambda =b_{1}m_{1} +\cdots +b_{l}m_{l} \in \Gamma_{1}$ and 
		$\mu = a_{1}n_{1} +\cdots +a_{k}n_{k} \in\Gamma_{2}$ be two positive 
		integers satisfying $\gcd(\lambda, \mu) = 1 $, with 
		$\lambda \notin \{m_{1}, \dots ,m_{l} \}$, $\mu \notin \{n_{1}, \ldots ,n_{k}\}$ 
		and $\{\mu m_{1},\ldots , \mu m_{l}\}\cap \{\lambda n_{1},\ldots , \lambda n_{k}\} = \phi$. 
		The numerical semigroup $\Gamma_{1}\#_{\lambda,\mu} \Gamma_{2}=\langle \mu  m_{1},\ldots , \mu m_{l}, \lambda n_{1},\ldots, \lambda n_{k}\rangle$ is called a \textit{gluing} of the semigroups $\Gamma_{1}$ and $\Gamma_{2}$ with 
		respect to $\mu $ and $\lambda$.
	}
\end{definition}

Section 3 in \cite{Sahin-Gluing} states that if $\Gamma$ be obtained by gluing 
$\Gamma_{1} = \Gamma(m_{1},\dots, m_{l})$ and 
$\Gamma_{2} = \Gamma(n_{1},\dots, n_{k})$ 
with respect to $\lambda = \sum_{i=1}^{l}b_{i}m_{i}$ and $\mu = \sum_{i=1}^{k}a_{i}n_{i}$,  
and if the defining ideals $I(\Gamma_{1}) \subset \mathbb{K}[x_{1},\ldots, x_{l}]$ and 
$I(\Gamma_{2}) \subset \mathbb{K}[y_{1},\ldots, y_{k}]$ are generated by the sets 
$G_{1} = \{f_{1},\ldots, f_{d}\} $ and $G_{2} = \{g_{1},\ldots, g_{r}\}$ respectively, 
then the defining ideal 
$I(\Gamma)\subset R = \mathbb{K}[x_{1},\ldots, x_{l}, y_{1},\ldots, y_{k}]$ is generated by the set 
$G = G_{1}\cup G_{2}\cup \{\rho\}$, where 
$\rho=x_{1}^{b_{1}}\dots x_{l}^{b_{l}}-y_{1}^{a_{1}}\dots y_{k}^{a_{k}}$.

\medspace	

	\begin{theorem}\label{Gluing-Trace}
		\begin{enumerate}[(i)]
			\item \label{(i)} $\mathrm{tr}(\Gamma_1 \#_{\lambda,\mu}\Gamma_2)=\mu \mathrm{tr}(\Gamma_1)+\lambda \mathrm{tr}(\Gamma_2)$.
			
			\item 	$\mathrm{res}(\Gamma_1 \#_{\lambda,\mu}\Gamma_2)=\mu \mathrm{res}(\Gamma_1)+ \lambda \mathrm{res}(\Gamma_2)$.
		\end{enumerate}
	
	\end{theorem}
	
	\begin{proof}
	Let $ \mathbb{F}: 0 \rightarrow  F_k \xrightarrow {\phi_{k}} \cdots \xrightarrow {\phi_2} F_1  \xrightarrow {\phi_1} F_0  $ be a minimal $\Gamma_1$-graded free resolution of $I_{\Gamma_1}$ with $H_0= R/I_{\Gamma_1}$, $ \mathbb{G}: 0 \rightarrow  G_l \xrightarrow {\psi_{k}} \cdots \xrightarrow {\psi_2} G_1  \xrightarrow {\psi_1} G_0  $ be a minimal $\Gamma_2$-graded free resolution of $I_{\Gamma_1}$ with $H_0= R/I_{\Gamma_2}$ and $\mathbb{H}_{\rho}: 0 \rightarrow R \xrightarrow {\rho} R \rightarrow 0  $. From \cite[Theorem 3.2]{Sahin-Gluing}  we know that $\mathbb{H}_{\rho} \otimes \mathbb{F} \otimes \mathbb{G}$ give a minimal free resolution of $I_{\Gamma_1 \#_{\lambda,\mu} \Gamma_2 }$. By \cite[Corollary 3.9]{Bruns-Herzog}, the dual complex $(\mathbb{H}_{\rho} \otimes \mathbb{F} \otimes \mathbb{G})^{\ast}$ is a minimal free resolution of $\omega_{\mathbb{K}[{\Gamma_1 \#_{\lambda,\mu} \Gamma_2 }]}$. From the isomorphism $\mathbb{H}_{\rho}^{\ast} \otimes \mathbb{F}^{\ast} \otimes \mathbb{{G}}^{\ast} \cong (\mathbb{H}_{\rho} \otimes \mathbb{F} \otimes \mathbb{G})^{\ast}$ we can deduce that $\mathbb{H}_{\rho}^{\ast} \otimes \mathbb{F}^{\ast} \otimes \mathbb{{G}}^{\ast}$ is exact and it resolves $\omega_{\rho R} \otimes \omega_{\mathbb{K}[\Gamma_1]} \otimes \omega_{\mathbb{K}[\Gamma_2]}$ minimally. By \cite[Proposition 4.1]{Trace}, we have $\mathrm{tr}(\omega_{\mathbb{K}[\Gamma_1 \#_{\lambda,\mu} \Gamma_2  ]})=\mathrm{tr}(\omega_{\rho R}) \cdot \mathrm{tr}(\omega_{\mathbb{K}[\Gamma_1]} \cdot \mathrm{tr} ( \omega_{\mathbb{K}[\Gamma_2]})$, and hence the exponent of monomials in $\mathrm{tr}(\omega_{\mathbb{K}[\Gamma_1 \#_{\lambda,\mu} \Gamma_2  ]})$ are of the form $ \mu \mathrm{tr}(\Gamma_1)+\lambda \mathrm{tr}(\Gamma_2)$, so semigroup ideal $ \mathrm{tr}(\Gamma_1 \#_{\lambda,\mu}\Gamma_2)=\mu \mathrm{tr}(\Gamma_1)+\lambda \mathrm{tr}(\Gamma_2)$. Note that $\mathbb{K}[H]/\mathrm{tr}(\omega_{\mathbb{K}[H]}))$ forms a vector space over $\mathbb{K}$ with a $\mathbb{K}$-basis $\{t^h :H \setminus \mathrm{tr}(H)\}$ for any numerical semigroup $H$ and $\mathrm{dim}_{\mathbb{K}}\mathbb{K}[H]/\mathrm{tr}(\omega_{\mathbb{K}[H]}))=|H \setminus \mathrm{tr}(H)|$. Consider a vector space homomorphism $\eta :\mu \cdot \mathbb{K}[\Gamma_1]/\mathrm{tr}(\omega_{\mathbb{K}[\Gamma_1]})) \oplus \lambda \cdot \mathbb{K}[\Gamma_2]/\mathrm{tr}(\omega_{\mathbb{K}[\Gamma_2]})) \rightarrow \mathbb{K}[\Gamma_1 \#_{\lambda,\mu}\Gamma_2]/\mathrm{tr}(\omega_{\mathbb{K}[\Gamma_1 \#_{\lambda,\mu}\Gamma_2]}))$ defined by $\eta(\mu t^{ s_1}+\lambda t^{s_2})=t^{\mu s_1+\lambda s_2}$. From (\ref{(i)}), it is clear that $\eta$ is a vector space isomorphism, and  $\mathrm{dim}_{\mathbb{K}}\frac{\mathbb{K}[\Gamma_1 \#_{\lambda,\mu}\Gamma_2]}{\mathrm{tr}(\omega_{\mathbb{K}[\Gamma_1 \#_{\lambda,\mu}\Gamma_2]}))}=\mu \mathrm{dim}_{\mathbb{K}}\frac{\mathbb{K}[\Gamma_1] }{\mathrm{tr}(\omega_{\mathbb{K}[\Gamma_1]}}+ \lambda \mathrm{dim}_{\mathbb{K}}\frac{\mathbb{K}[\Gamma_2]}{\mathrm{tr}(\omega_{\mathbb{K}[\Gamma_2]})}$. Therefore, $\mathrm{res}(\Gamma_1 \#_{\lambda,\mu}\Gamma_2)=\mu \mathrm{res}(\Gamma_1)+ \lambda \mathrm{res}(\Gamma_2)$.
	\end{proof}
	
	\smallskip
	
	\begin{corollary}
		Let $\mathbb{K}[\Gamma_{1}]$ and $\mathbb{K}[\Gamma_{2}]$ are two nearly Gorenstein rings. Then $\mathbb{K}[\Gamma]$ is never nearly Gorenstein, where $\Gamma$ is a gluing of $\Gamma_1$ and $\Gamma_2$.
	\end{corollary}
	\begin{proof}
		The proof follows directly from Theorem \ref{Gluing-Trace} and the fact that $\mathbb{K}[\Gamma]$ is nearly Gorenstein if and only if $\mathrm{res}(\Gamma)\leq 1$.
	\end{proof}
	
	\smallskip
	
	\begin{lemma}\label{PF-Gluing}
		Assuming the notation of current section,
		\begin{enumerate}[(i)]
			\item $\mathrm{PF}(\Gamma_1 \#_{\lambda,\mu}\Gamma_2)=\{\mu g +\lambda g'+ \mu \lambda : g \in \mathrm{PF}(\Gamma_1), g' \in \mathrm{PF}(\Gamma_2)\}$.
			\item $\mathrm{F}(\Gamma_1 \#_{\lambda,\mu}\Gamma_2)=\mu  \mathrm{F}(\Gamma_1) +\lambda \mathrm{F}(\Gamma_2)+ \mu \lambda$.
		\end{enumerate}
	\end{lemma}
	
	\begin{proof}
		See Proposition 6.6 in \cite{PF-Gluing}.
	\end{proof}
	
	\medspace
	
Set $C_H=t^{\mathrm{F}(H)}\omega_{\mathbb{K}[H]}=\sum_{\alpha \in \mathrm{PF}(H)} \mathbb{K}[H]t^{\mathrm{F}(H)-\alpha}$ for any numerical semigroup $H$. Then we have $\mathbb{K}[H] \subseteq C_H \subseteq \overline{\mathbb{K}[H]}$, where
$\overline{\mathbb{K}[H]}$ denotes the integral closure of $\mathbb{K}[H]$. By Lemma \ref{PF-Gluing}, we have
$C_{\Gamma_1 \#_{\lambda,\mu}\Gamma_2}=t^{\mathrm{F}(\Gamma_1 \#_{\lambda,\mu}\Gamma_2)}\omega_{\mathbb{K}[{\Gamma_1 \#_{\lambda,\mu} \Gamma_2 }]}= \sum_{g\in \mathrm{PF}(\Gamma_1),g' \in \mathrm{PF}(\Gamma_2)}\mathbb{K}[{\Gamma_1 \#_{\lambda,\mu} \Gamma_2 }]t^{\mu \mathrm{F}(\Gamma_1)-\mu g+ \lambda \mathrm{F}(\Gamma_2)-\lambda g'}$. Hence $\mathrm{dim}_{\mathbb{K}} \frac{C_{\Gamma_1 \#_{\lambda,\mu}\Gamma_2}}{\mathbb{K}[\Gamma_1 \#_{\lambda,\mu}\Gamma_2]}=\mu \mathrm{dim}_{\mathbb{K}}\frac{C_{\Gamma_1}}{\mathbb{K}[\Gamma_1]}+\lambda \mathrm{dim}_{\mathbb{K}}\frac{C_{\Gamma_2}}{\mathbb{K}[\Gamma_2]}$. By \cite[Lemma 2.2]{Upper bound}, $g(\Gamma_1\#_{\lambda,\mu} \Gamma_2)-n(\Gamma_1\#_{\lambda,\mu} \Gamma_2)=l_{\mathbb{K}[\Gamma_1\#_{\lambda,\mu} \Gamma_2]}(\frac{C_{\Gamma_1\#_{\lambda,\mu} \Gamma_2}}{\mathbb{K}[\Gamma_1\#_{\lambda,\mu} \Gamma_2]})=\mu l_{\mathbb{K}[\Gamma_1\#_{\lambda,\mu} \Gamma_2]}(\frac{C_{\Gamma_1}}{\mathbb{K}[\Gamma_1]})+
\lambda l_{\mathbb{K}[\Gamma_1\#_{\lambda,\mu} \Gamma_2]}(\frac{C_{\Gamma_2}}{\mathbb{K}[\Gamma_2]})$. Now if $\Gamma_1,\Gamma_2$ satisfies Question \ref{Ques}, then $\mathrm{res}(\Gamma_1 \#_{\lambda,\mu} \Gamma_2)=\mu \mathrm{res}(\Gamma_1)+\lambda \mathrm{res}(\Gamma_2)\leq \mu(g(\Gamma_1)-n(\Gamma_1))+\lambda(g(\Gamma_2)-n(\Gamma_2))=g(\Gamma_1 \#_{\lambda,\mu}\Gamma_2)-n(\Gamma_1 \#_{\lambda,\mu}\Gamma_2)$, hence $\Gamma_1 \#_{\lambda,\mu}\Gamma_2$ also satisfies Question \ref{Ques}.

\medspace

	\subsection{Residue of lifting of monomial curve}\label{lift}
	
	Let $\Gamma$ be a numerical semigroup minimally generated by $n_1 < \dots <n_r$. By a $k$-lifting $S_k$ of $\Gamma$ we mean
	the numerical semigroup minimally generated by $n_1,kn_2,\dots,kn_r$, where $k$ is a positive integer with $\mathrm{gcd}(k, n_1) = 1$. 
	Let $A=\mathbb{K}[x_1,\dots,x_r]$ and $B=\mathbb{K}[y_1,\dots,y_r]$, and we set $\mathrm{degree}(x_i)=n_i$ for $ 1 \leq i \leq r, \mathrm{deg}(y_i)=kn_i$ for $2 \leq i \leq r-1$ and $\mathrm{deg}(y_1)=n_1$. In this subsection denote semigroup rings by $R=\mathbb{K}[\Gamma] $ and $R_k=\mathbb{K}[\Gamma_k] $
	
	\begin{theorem}\label{Lift-Can}
		$\mathrm{tr}(\Gamma_k)=k\cdot\mathrm{tr}(\Gamma)$
	\end{theorem}
	\begin{proof}
		Let $$
		\mathbb{F}: 0 \rightarrow F_p \rightarrow F_{p-1} \rightarrow \dots \rightarrow F_0 \rightarrow A/I(\Gamma) \rightarrow 0
		$$
		 be a minimal free $A$-resolution of $A/I(\Gamma)$.
		 Then by \cite[Corollary 3.9]{Bruns-Herzog}, the dual complex $\mathbb{F}^{\ast}=\mathrm{Hom}_A(\mathbb{\mathbb{F}'},A)$ is a minimal free resolution of $\omega_{\mathbb{K}[\Gamma]}$. 
		 
		 As indicated in \cite{Var}, a minimal $\Gamma_k$-graded free resolution $\mathbb{F}_k$ of $\mathbb{K}[\Gamma_k]$ is obtained from a
		 minimal $\Gamma$-graded free resolution of $\mathbb{K}[\Gamma]$ via the faithfully flat extension $f : A \rightarrow B$,
		 defined by sending $x_1 \rightarrow y_1^k$ 
		 and $x_i \rightarrow y_i$ for all $i > 1$, and the dual complex $\mathbb{F}_k^{\ast} =\mathrm{Hom}_B(\mathbb{F}_k,B)=\mathrm{Hom}_B(\mathbb{F}\otimes_A B,B)\cong \mathrm{Hom}_A(\mathbb{F},A)\otimes_A B \cong \mathbb{F}^{\ast} \otimes_A B $.  Hence $\omega_{\mathbb{K}[\Gamma_k]}=k\omega_{\mathbb{K}[\Gamma]}$. It is clear that any $\phi \in \mathrm{Hom}_B(\omega_{\mathbb{K}[\Gamma_k]},B)$ can be obtained by extending the map $\psi \in \mathrm{Hom}_A(\omega_{\mathbb{K}[\Gamma]},A)$ via $f$. Hence $\mathrm{tr}_B(\omega_{\mathbb{K}[\Gamma_k]})=\sum_{\phi\in \mathrm{Hom}_B(\omega_{\mathbb{K}[\Gamma_k]},B)}\phi(\omega_{\mathbb{K}[\Gamma_k]})=k\sum_{\psi\in \mathrm{Hom}_A(\omega_{\mathbb{K}[\Gamma],A)}}\psi(\omega_{\mathbb{K}[\Gamma]})=k \cdot \mathrm{tr}_{A}(\omega_{\mathbb{K}[\Gamma])}$.  
	\end{proof}
	
	\begin{corollary}\label{Lift-Res}
		$\mathrm{res}(\Gamma_k)=k \cdot \mathrm{res}(\Gamma)$.
	\end{corollary}
	\begin{proof}
		Note that $\mathbb{K}[\Gamma]/\mathrm{tr}(\omega_{\mathbb{K}[\Gamma]})$ is a finite dimensional vector space with a $\mathbb{K}$-basis given by $\{t^\gamma:\gamma \in \Gamma\setminus \mathrm{tr}(\Gamma)\}$. Consider a map $\eta: k \cdot \mathbb{K}[\Gamma]/\mathrm{tr}(\omega_{\mathbb{K}[\Gamma]}) \rightarrow \mathbb{K}[\Gamma_k]/\mathrm{tr}(\omega_{\mathbb{K}[\Gamma_k]})$ defined by $\eta(kt^h)=t^{kh}$. Due to Theorem \ref{Lift-Can}, $\eta$ is well defined and it is easy to check that $\eta$ is bijective homomomorphism. Hence
		$\dim_{\mathbb{K}}(\mathbb{K}[\Gamma_k]/\mathrm{tr}(\omega_{\mathbb{K}[\Gamma_k]}))=k \cdot \dim_{\mathbb{K}}(\mathbb{K}[\Gamma]/\mathrm{tr}(\omega_{\mathbb{K}[\Gamma]}))$ and  we have the required result.
	\end{proof}
	
	One direct consequence of Corollary \ref{Lift-Res} is as following
	
	\begin{corollary}
		If $\Gamma$ is nearly Gorenstein semigroup then $\Gamma_k$ be never nearly Gorenstein for $k\geq 2$.
	\end{corollary}

	Let $C= t^{F(\Gamma)}\omega_{R}=\sum_{\alpha \in \mathrm{PF}(S)}Rt^{F(\Gamma)-\alpha}$. Then we obtained that $R\subseteq C \subseteq \overline{R}=\mathbb{K}[t]$, where $\overline{R}$ denotes the integral closure of $R$. By  \cite{Sahin}, we know that $\mathrm{PF}(\Gamma_k)=\{kf+(k-1)n_1 \mid f \in \mathrm{PF}(\Gamma)\}$ and $F(\Gamma_k)=kF(\Gamma)+(k-1)n_1$, hence $C_k=\sum_{\alpha \in \mathrm{PF}(\Gamma)}R_kt^{k(F(\Gamma)-\alpha)}$ and $R_k \subseteq C_k \subseteq \overline{R_k}=\mathbb{K}[t]$. Now consider  a module homomorphism  $\psi: k\frac{C}{R} \rightarrow \frac{C_k}{R_k}$ defined by $\psi(k(c+R))=kc+R_k$. It is clear that $\psi$ is a module isomorphism and hence $l_{R_{k}}(\frac{C_k}{R_k})=kl_R(\frac{C}{R})$. By \cite[Lemma 2.2]{Upper bound}, $g(\Gamma_k)-n(\Gamma_k)=l_{R_{k}}(\frac{C_k}{R_k})=k l_R(\frac{C}{R})=k(g(\Gamma)-n(\Gamma))$. Now if $\Gamma$ satisfies Question \ref{Ques}, then $\mathrm{res}(\Gamma_k)=k\mathrm{res}(\Gamma)\leq k(g(\Gamma)-n(\Gamma))=g(\Gamma_k)-n(\Gamma_k)$, hence $\Gamma_k$ also satisfies Question \ref{Ques}.

\medspace

	\bibliographystyle{amsalpha}

\end{document}